\documentclass[letterpaper, 10 pt, conference]{ieeeconf}  % Comment this line out if you need a4paper
\usepackage{etex}
\reserveinserts{28}

\IEEEoverridecommandlockouts                              % This command is only needed if
\usepackage{cite}
\usepackage{amsmath}
\usepackage{amssymb}
\usepackage{balance}
\usepackage{graphicx}
\usepackage{subfigure}
\usepackage{epstopdf}
\usepackage{amsfonts}
\usepackage[all]{xy}
\usepackage{array,xcolor}
\usepackage{color}
\usepackage{bbm}
\usepackage{dsfont}
\usepackage{caption}
\usepackage{booktabs}
\usepackage{multirow}

\usepackage[font=footnotesize,labelfont=bf]{caption}

\usepackage{tikz}
\usetikzlibrary{shapes,matrix,decorations.shapes}
\usetikzlibrary{arrows,decorations.pathmorphing,backgrounds,positioning,fit,matrix}
\usetikzlibrary{shapes,arrows,chains}
\usetikzlibrary[calc]
\usetikzlibrary[plotmarks]
\usetikzlibrary[patterns,decorations]

\newcolumntype{C}[1]{>{\centering}m{#1}}
\newtheorem{definition}{Definition}
\newtheorem{theorem}{Theorem}

\newtheorem{proposition}{Proposition}
\newtheorem{remark}{Remark}

\newtheorem{example}{Example}

\newcommand{\tr}{{{\mathsf T}}}
\newcommand{\PSD}{\mbox{\it PSD}}
\newcommand{\SOS}{\mbox{\it SOS}}
\newcommand{\SSOS}{\mbox{\it SSOS}}
\newcommand{\DSOS}{\mbox{\it DSOS}}
\newcommand{\SDSOS}{\mbox{\it SDSOS}}

\usepackage{hyperref}
\hypersetup{colorlinks=true,
			citecolor=blue,
			linkcolor=black,
			urlcolor=blue}

\title{\LARGE \bf
Sparse sum-of-squares (SOS) optimization: A bridge between DSOS/SDSOS and SOS optimization for sparse polynomials
}

\author{Yang Zheng$^{\dagger}$, Giovanni Fantuzzi$^{\ddagger}$, and Antonis~Papachristodoulou$^{\dagger}$% <-this % stops a space
\thanks{Y. Zheng is supported by the Clarendon and the Jason Hu Scholarships. G. Fantuzzi is supported by an EPSRC Doctoral Prize Fellowship. The authors would like to thank Dr. Aivar Sootla at the University of Oxford for many fruitful discussions.% especially on factor-width $k$ matrices.
}
\thanks{$^{\dagger} $Y. Zheng, and A. Papachristodoulou are with Department of Engineering Science at the University of Oxford. (E-mail:        \{yang.zheng, antonis\}@eng.ox.ac.uk)}%
\thanks{$^{\ddagger}$G. Fantuzzi is with Department of Aeronautics, Imperial College London, South Kensington Campus. (E-mail:        gf910@ic.ac.uk)}
}

\begin{document}

\maketitle
\thispagestyle{empty}
\pagestyle{empty}

%%%%%%%%%%%%%%%%%%%%%%%%%%%%%%%%%%%%%%%%%%%%%%%%%%%%%%%%%%%%%%%%%%%%%%%%%%%%%%%%
\begin{abstract}
	 Optimization over non-negative polynomials is fundamental for nonlinear systems analysis and control. We investigate the relation between three tractable relaxations for  optimizing over \textit{sparse} non-negative polynomials: sparse sum-of-squares (SSOS) optimization, diagonally dominant sum-of-squares (DSOS) optimization, and scaled diagonally dominant sum-of-squares (SDSOS) optimization. We prove that the set of SSOS polynomials, an inner approximation of the cone of SOS polynomials, strictly contains the spaces of sparse DSOS/SDSOS polynomials. When applicable, therefore, SSOS optimization is less conservative than its DSOS/SDSOS counterparts. Numerical results for large-scale sparse polynomial optimization problems demonstrate this fact, and also that SSOS optimization can be faster than DSOS/SDSOS methods despite requiring the solution of semidefinite programs instead of less expensive linear/second-order cone programs.
\end{abstract}

%%%%%%%%%%%%%%%%%%%%%%%%%%%%%%%%%%%%%%%%%%%%%%%%%%%%%%%%%%%%%%%%%%%%%%%%%%%%%%%%
\section{Introduction}

Optimization over non-negative polynomials plays a fundamental role in analysis and control of systems with polynomial dynamics. For instance, the construction of polynomial Lyapunov or Lyapunov-type functions subject to suitable polynomial inequalities can prove nonlinear stability of equilibrium solutions~\cite{papachristodoulou2002construction}, approximate basins of attraction~\cite{Tan2006a}, and provide bounds on infinite-time averages~\cite{Chernyshenko2014a,Goluskin2016e,Fantuzzi2016b}.

Since deciding whether a given polynomial $p(x)$ is non-negative is NP-hard in general, a popular alternative is to look for a decomposition of $p(x)$ as a sum-of-squares (SOS) of polynomials with lower degree. Checking this sufficient condition for non-negativity is attractive because it amounts to solving a semidefinite program (SDP)~\cite{parrilo2003semidefinite,lasserre2001global}, a well-known type of convex optimization problem for which polynomial-time algorithms exist. However, the dimension of this SDP typically grows in a combinatorial fashion as the number of variables and the polynomial degree increase~\cite{parrilo2003semidefinite}, and very large SDPs are required to solve even when one employs well-known dimension reduction techniques, such as the Newton polytope~\cite{reznick1978extremal}, diagonal inconsistency~\cite{lofberg2009pre}, and symmetry~\cite{gatermann2004symmetry} or facial reduction~\cite{permenter2014basis}. Consequently, SOS-based analysis is only practical for polynomial dynamical systems with few states and/or low degree.

In order to improve scalability, it has been proposed by many authors to replace positivity certificates based on SOS representations with other sufficient conditions for non-negativity, which are stronger but have a lower computational complexity. For correlatively sparse polynomials, characterized by sparse couplings between different independent variables, Waki \textit{et al.}~\cite{waki2006sums} proposed to look for a decomposition as a sum of SOS polynomials, each involving only small subsets of the independent variables. While being  more restrictive, the search for such a sparse sum-of-squares (SSOS) decomposition can be carried out at a fraction of the computational cost required for a standard SOS decomposition, because an SDP with one large matrix variable is replaced with an SDP with multiple much smaller matrix variables. The latter can be solved more efficiently, a fact that also underpins the more recent sparse-BSOS~\cite{weisser2018sparse} and multi-ordered Lasserre relaxation hierarchies~\cite{josz2018lasserre} for sparse polynomial optimization. Two other alternatives to SOS optimization, applicable also to polynomials without correlative sparsity, were put forward by Ahmadi and Majumdar~\cite{ahmadi2017dsos}, who observed that the cones of \emph{diagonally dominant sum-of-squares} (DSOS) polynomials and of \emph{scaled diagonally dominant sum-of-squares} (SDSOS) polynomials are strict subsets of the cone of SOS polynomials. Optimization problems over DSOS and SDSOS polynomials can be recast as linear programs (LPs) and second-order cone programs (SOCP), respectively, and both of these can be solved with algorithms that scale more favourably than those for SDPs~\cite{ahmadi2017dsos}. On the other hand, DSOS/SDSOS optimization might be very conservative (although the conservatism may be reduced using iterative methods based on basis pursuit~\cite{ahmadi2015sum} or column generation~\cite{ahmadi2017optimization}).

The availability of such a variety of approaches poses a simple but important dilemma: when more than one method can be applied, which one should be used? To answer this question, theoretical results comparing the degree of conservatism and computational complexity for each of the aforementioned approaches would be desirable. In this paper, therefore, %we take a first step in this direction and
we study the relation between DSOS/SDSOS/SSOS positivity certificates for polynomials with correlative sparsity. Specifically, we prove that if a DSOS/SDSOS decomposition exists \textit{and} the correlative sparsity is \textit{chordal} (meaning that it can be represented by a chordal graph), then an SSOS decomposition is also available. In other words, the cones of DSOS/SDSOS polynomials with chordal correlative sparsity are strictly contained within the cone of polynomials that admit an SSOS decomposition in the sense of Waki \textit{et al.}~\cite{waki2006sums}. Thus, for polynomials with chordal correlative sparsity, DSOS/SDSOS optimization is provably more conservative than SSOS optimization (at least when the iterative improvement techniques for DSOS/SDSOS optimization mentioned above are not utilised). %{\color{red}[Again, this should be mentioned for fairness.]}.
Also, SSOS optimization promises better scalability compared to standard SOS optimization. Therefore, we argue that SSOS optimization is a suitable candidate to bridge the gap between DSOS/SDSOS and SOS optimization. 

The rest of this paper is organized as follows. Section~\ref{Section:Preliminaries} reviews some basic facts about SOS/DSOS/SDSOS polynomials, useful notions from graph theory, and correlatively sparse polynomials. We give a new interpretation of the positivity certificates of Waki \textit{et al}~\cite{waki2006sums} in Section~\ref{Section:SparseSOSWaki}, which enables the connection to DSOS/SDSOS conditions for correlatively sparse polynomials in Section~\ref{Section:ConnectingSSO-SDSOS}. In Section~\ref{Section:MatrixAnalysis}, we extend our analysis to the cones of sparse DSOS/SDSOS/SSOS polynomial matrices.
Section~\ref{Section:Results} presents numerical results. Finally, Section~\ref{Section:Conclusion} concludes the paper. 

%%%%%%%%%%%%%%%%%%%%%%%%%%%%%%%%%%%%%%%%%%
\section{Preliminaries}
\label{Section:Preliminaries}

%%%%%%%%%%%
\subsection{SOS, DSOS, and SDSOS polynomials}
\label{Section:SOS-DSOS-SDSOS}

Given a vector of variables $x \in \mathbb{R}^n$ and a multi-index $\alpha \in \mathbb{N}^n$, the quantity $x^{\alpha} := \prod_{i=1}^{n}x_i^{\alpha_i}$ is a monomial of degree $\vert \alpha \vert := \sum_{i=1}^n \alpha_i$. For $d \in \mathbb{N}$, we let $\mathbb{N}^n_d = \{\alpha \in \mathbb{N}^n: \vert \alpha \vert \leq d\}$. An $n$-variate polynomial of degree $2d$ can be written as $p(x) = \sum_{ \alpha \in \mathbb{N}^n_{2d}} c_{\alpha} x^{\alpha}$. We denote the set of polynomials in $n$ variables with real coefficients of degree no more than $2d$ by $\mathbb{R}[x]_{n,2d}$, and the subset of nonnegative polynomials in $\mathbb{R}[x]_{n,2d}$ by $\PSD_{n,2d}$.

Checking if $p(x) \in \PSD_{n,2d}$ is NP-hard already for polynomials of degree 4, %except the cases where $(n,d) = $ are small,
but it is computationally tractable to test membership to the following subsets of $\PSD_{n,2d}$.

\textit{(1) SOS polynomials:} A polynomial $p(x) \in \mathbb{R}[x]_{n,2d}$ is an SOS polynomial if there exist $f_i \in \mathbb{R}[x]_{n,d}$, $i = 1,\,\ldots,\,s$ such that
\begin{equation*}
p(x) = \sum_{i=1}^s f^2_i(x).
\end{equation*}
We denote the set of $n$-variate SOS polynomials of degree no larger than $2d$ by $\SOS_{n,2d}$.
It is known~\cite{parrilo2003semidefinite} that $p(x) \in \SOS_{n,2d}$ if and only if there exists a PSD matrix $Q$ (denoted $Q\succeq 0$), such that
\begin{equation} \label{E:MatrixQ}
p(x) = v_d(x)^\tr \, Q \, v_d(x),
\end{equation}
where
$%\begin{equation} %\label{E:MonomialVector}
v_d(x) = [ 1,x_1,x_2,\ldots,x_n,x_1^2,x_1x_2,\ldots,x_n^d ]^\tr
$ %\end{equation}
is the vector of monomials of $x$ of degree $d$ or less. Following~\cite{parrilo2003semidefinite}, we refer to~\eqref{E:MatrixQ} as the \emph{Gram matrix} representation of the polynomial $p(x)$. Note that the size of the Gram matrix $Q$ is ${n+d \choose d} \times {n+d \choose d}$ in general.

\textit{(2) DSOS polynomials:} Recall that a symmetric matrix $A \in \mathbb{S}^r$ is diagonally dominant (DD) if $A_{ii} \geq \sum_{j=1}^r |A_{ij}|$ for all $i = 1,\,\ldots,\,r$, and that DD matrices are positive semidefinite (this directly follows, for example, from Gershgorin's circle theorem). Following~\cite{ahmadi2017dsos}, we say that polynomial $p(x) \in \mathbb{R}[x]_{n,2d}$ is a \textit{diagonally dominant sum-of-squares} (DSOS) if it admits a Gram matrix representation~\eqref{E:MatrixQ} with a DD Gram matrix $Q$. We denote the set of DSOS polynomials in $n$ variables and degree no larger than $2d$ by $\DSOS_{n,2d}$.

\textit{(3) SDSOS polynomials:} Recall that a symmetric matrix $A \in \mathbb{S}^r$ is scaled diagonally dominant (SDD) if there exists a positive definite $r\times r$ diagonal matrix $D$ such that $DAD$ is diagonally dominant. Following~\cite{ahmadi2017dsos}, we say that polynomial $p(x) \in \mathbb{R}[x]_{n,2d}$ is a \textit{scaled diagonally dominant sum-of-squares} (SDSOS) if it admits a Gram matrix representation~\eqref{E:MatrixQ} with an SDD Gram matrix $Q$. We denote the set of SDSOS polynomials in $n$ variables and degree no larger than $2d$ by $\SDSOS_{n,2d}$.

Let {$p_0,\,\ldots,\,p_t \in \mathbb{R}[x]_{n,2d}$} be given polynomials. An SOS optimization problem takes the standard form
\begin{equation}
\label{E:generalSOS}
    \begin{aligned}
        \min_{u}\quad & w^\tr u \\[-1ex]
        \text{subject to} \quad  & p_0(x) + \sum_{i=1}^t u_ip_i(x) \in \SOS_{n,2d},
    \end{aligned}
\end{equation}
where $u \in \mathbb{R}^t$ is the decision variable. It is not difficult to see that~\eqref{E:MatrixQ} enables one to recast~\eqref{E:generalSOS} as an SDP~\cite{parrilo2003semidefinite}. It has also been proved that if $\SOS_{n,2d}$ is replaced with $\DSOS_{n,2d}$ in~\eqref{E:generalSOS} (resp. $\SDSOS_{n,2d}$), then one obtains an LP (resp. SOCP)~\cite{ahmadi2017dsos}. %In principle, when one moves from SOS optimization to SDSOS/DSOS optimization, the ability of scalability increases and the quality of solutions decreases.
Thus, DSOS/SDSOS optimization are more scalable alternatives to SOS optimization, but are typically more conservative since the strict inclusion $\DSOS_{n,2d}\subset \SDSOS_{n,2d} \subset \SOS_{n,2d}$ holds.

%%%%%%%%%%%
\subsection{Chordal graphs and sparse matrices}

%\subsubsection{Chordal graphs}

The sparsity of a polynomial $p(x) \in \mathbb{R}[x]_{n,2d}$, or of any of its possible Gram matrices, can be conveniently represented by an undirected graph $\mathcal{G}(\mathcal{V},\mathcal{E})$, that is, a set of nodes $\mathcal{V}=\{1,2,\dots, n\}$ and a set of edges $\mathcal{E} \subseteq \mathcal{V} \times \mathcal{V}$ such that $(i,j) \in \mathcal{E}$ implies that $(j,i) \in \mathcal{E}$ also. For this reason, it is useful to review some essential notions about graphs and their relation to sparse matrices.

A \emph{cycle} of length $k$ of an undirected graph $\mathcal{G}(\mathcal{V},\mathcal{E})$ is a sequence of nodes $\{v_1, v_2, \ldots, v_k\} \subseteq \mathcal{V}$ such that $(v_k, v_{1}) \in \mathcal{E}$ and $(v_i, v_{i+1}) \in \mathcal{E}$ for all $i = 1,\,\ldots,\,k-1$. A \emph{chord} in a cycle is an edge $(v_i,v_j)$ between nonconsecutive nodes of the cycle, and an undirected graph is called \emph{chordal} if all its cycles with length $\geq 4$ have a chord. Simple examples of chordal graphs are given in Fig.~\ref{F:ChordalGraph}. Note that any non-chordal graph $\mathcal{G}(\mathcal{V},\mathcal{E})$ can always be extended to a chordal graph $\mathcal{G}(\mathcal{V},\hat{\mathcal{E}})$ by adding edges to $\mathcal{E}$~\cite{vandenberghe2014chordal}.
Finally, a \emph{clique} $\mathcal{C} \subseteq \mathcal{V}$ is a subset of nodes where $(i,j) \in \mathcal{E}, \forall i,j \in \mathcal{C}, i \neq j$. If a clique $\mathcal{C}$ is not included in any other clique, then it is referred to as a \emph{maximal clique}. For example, the graph in Fig.~\ref{F:ChordalGraph}(c) has maximal cliques $\{1, 2, 3\}$ and $\{2, 3, 4\}$, while that in Fig.~\ref{F:ChordalGraph}(b) has maximal cliques $\{1, i+1\}$, $i = 1,\,\ldots,\,4$.

\begin{figure}[t]
	\centering
	\footnotesize
	\begin{tikzpicture}%[every node/.style={text height=1ex, text width=1em, text centered, align=center}]
	\matrix (m) [matrix of nodes,
	row sep = 0.8em,	
	column sep = 1em,	
	nodes={circle, draw=black}] at (-2.8,0)
	{ &  & \\ 1 & 2 & 3 \\ & &\\};
	\draw (m-2-1) -- (m-2-2);
	\draw (m-2-2) -- (m-2-3);
	\node at (-2.8,-1.3) {(a)};
	\matrix (m2) [matrix of nodes,
	row sep = 0.8em,	
	column sep = 1em,	
	nodes={circle, draw=black}] at (0,0)
	{ & 3 & \\ 2 & 1 & 4 \\& 5 &\\};
	%{ & 5 & \\ 1 &4 & 3 \\&2&\\};
	\draw (m2-1-2) -- (m2-2-2);
	\draw (m2-2-1) -- (m2-2-2);
	\draw (m2-2-2) -- (m2-2-3);
	\draw (m2-2-2) -- (m2-3-2);
	\node at (0,-1.3) {(b)};
	
	\matrix (m3) [matrix of nodes,
	row sep = 0.8em,	
	column sep = 1.em,	
	nodes={circle, draw=black}] at (2.8,0)
	{ & 2 & \\ 1 &  & 4 \\& 3 &\\};
	\draw (m3-1-2) -- (m3-2-1);
	\draw (m3-1-2) -- (m3-2-3);
	\draw (m3-2-1) -- (m3-3-2);
	\draw (m3-2-3) -- (m3-3-2);
	\draw (m3-1-2) -- (m3-3-2);
	\node at (2.8,-1.3) {(c)};
	\end{tikzpicture}
	\caption{Examples of chordal graphs: (a) a line graph; (b) a star graph; (c) a triangulated graph.}
	\label{F:ChordalGraph}
\end{figure}
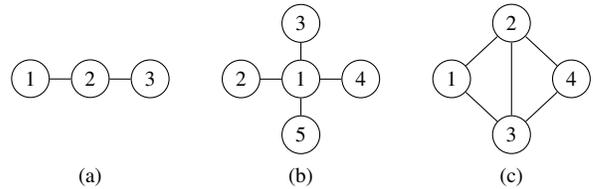

%\subsubsection{Matrix decomposition}
Given an undirected graph $\mathcal{G}(\mathcal{V},\mathcal{E})$, we consider the extended set of edges $\mathcal{E}^* = \mathcal{E} \cup\{(i,i), \forall i \in \mathcal{V}\}$ and define the space of $n \times n$ symmetric matrices with sparsity pattern characterized by $\mathcal{G}$ as
\begin{equation} \label{E:SparseSymMatrix}
\mathbb{S}^n(\mathcal{E},0) := \{X \in \mathbb{S}^{n} | X_{ij} = X_{ji}= 0\; \text{if} \; (j,i) \notin {\mathcal{E}}^* \}.
\end{equation}
Similarly the cone of sparse positive semidefinite (PSD) matrices with sparsity pattern described by $\mathcal{G}$ is
\begin{equation} \label{E:SparsePSD}
\mathbb{S}^n_+(\mathcal{E},0) := \{X \in \mathbb{S}^{n}(\mathcal{E},0) | X \succeq 0 \}.
\end{equation}
For simplicity, we will slight abuse the terminology and say that $\mathbb{S}^n(\mathcal{E},0)$ has sparsity pattern $\mathcal{E}$.

Finally, for each maximal clique $\mathcal{C}_k$ of $\mathcal{G}$ we define an index matrix $E_{\mathcal{C}_k} \in \mathbb{R}^{|\mathcal{C}_k| \times r}$ as
\begin{equation}
\label{E:IndexMatrix}
(E_{\mathcal{C}_k})_{ij} := \begin{cases} 1, \quad \text{if } \mathcal{C}_k(i) = j, \\ 0, \quad \text{otherwise}, \end{cases}
\end{equation}
where $|\mathcal{C}_k|$ denotes the number of nodes in $\mathcal{C}_k$, and $\mathcal{C}_k(i)$ denotes the $i$-th node in $\mathcal{C}_k$, sorted in the natural ordering. Then, $X_k = E_{\mathcal{C}_k}XE_{\mathcal{C}_k}^\tr \in \mathbb{S}^{|\mathcal{C}_k|}$ corresponds to the principal submatrix of $X$ defined by the indices in $\mathcal{C}_k$, while the operation $E_{\mathcal{C}_k}^\tr X_kE_{\mathcal{C}_k}$ ``inflates'' $X_k$ into a $\vert\mathcal{C}_k\vert \times \vert \mathcal{C}_k\vert $ matrix into a sparse $r\times r$ matrix.
The following key result states that when the graph $\mathcal{G}$ is chordal, then $X \in \mathbb{S}^r_{+}(\mathcal{E},0)$ if and only if it can be written as a sum of ``inflated'' PSD matrices.
\begin{theorem}[\!\!\cite{agler1988positive,griewank1984existence}]\label{T:ChordalDecompositionTheorem}
	Let $\mathcal{G}(\mathcal{V},\mathcal{E})$ be a chordal graph with maximal cliques $\mathcal{C}_1,\,\mathcal{C}_2,\,\ldots,\,\mathcal{C}_t$. Then, %$X\in\mathbb{S}^r_{+}(\mathcal{E},0)$ if and only if there exist $X_k \in \mathbb{S}^{\vert \mathcal{C}_k \vert}_+$, $k=1,\,\ldots,\,t$, such that
\end{theorem}
%\begin{equation*} %\label{E:DecompositionSparseCone}
%X = \sum_{k=1}^{t} E_{\mathcal{C}_k}^\tr X_k E_{\mathcal{C}_k}.
%\end{equation*}
\begin{equation*} %\label{E:DecompositionSparseCone}
X\in\mathbb{S}^r_{+}(\mathcal{E},0)
\quad \Leftrightarrow \quad
X = \sum_{k=1}^{t} E_{\mathcal{C}_k}^\tr X_k E_{\mathcal{C}_k}, \; X_k \in \mathbb{S}^{\vert \mathcal{C}_k \vert}_+.
\end{equation*}

%%%%%%%%%%%
\subsection{Correlatively sparse polynomials}

The notion of correlative sparsity was introduced by Waki \textit{et al.}~\cite{waki2006sums} to describe couplings between the variables $x_1,\,\ldots,\,x_n$ of a polynomial $p(x) = \sum_{\vert \alpha \vert \leq 2d} c_{\alpha} x^{\alpha}$. Key to this description is the so-called \textit{correlative sparsity matrix} (CSP matrix), a symmetric matrix ${\rm csp}(p) \in \mathbb{S}^n$ where
\begin{equation*}
[{\rm csp}(p)]_{ij} = \begin{cases}
1, \quad \text{if } i = j \text{ or } \exists \alpha \mid \alpha_i, \alpha_j \geq 1 \text{ and } c_{\alpha} \neq 0, \\
0, \quad \text{otherwise}. \\\end{cases}
\end{equation*}
For example, we have
\begin{equation*}
{\rm csp}(x_1^2 + x_2x_3^3)
= \begin{bmatrix}
1 & 0 & 0\\
0 & 1 & 1\\
0 & 1 & 1
\end{bmatrix}.
\end{equation*}

A polynomial  $p(x) \in \mathbb{R}[x]_{n,2d}$ is said to have a correlative sparsity pattern characterized by an undirected graph $\mathcal{G}(\mathcal{V},\mathcal{E})$ if ${\rm csp}(p)\in\mathbb{S}^n(\mathcal{E},0)$. It is then natural to define the vector space of polynomials with the same correlative sparsity as
\begin{equation*}
\mathbb{R}[x]_{n,2d}(\mathcal{E}) := \{ p\in \mathbb{R}[x]_{n,2d} \mid {\rm csp}(p) \in \mathbb{S}^n(\mathcal{E},0)\},
\end{equation*}
and its subset of sparse SOS polynomials as
\begin{equation*}
\SOS_{n,2d}(\mathcal{E}) := \mathbb{R}[x]_{n,2d}(\mathcal{E}) \cap \SOS_{n,2d}.
\end{equation*}
Throughout this paper, we assume that the correlative sparsity pattern $\mathcal{E}$ is chordal, or that a suitable chordal extension has been found.

%%%%%%%%%%%%%%%%%
\section{Revisiting sparse SOS decompositions}
\label{Section:SparseSOSWaki}

Suppose we wish to determine whether a correlatively sparse polynomial $p(x) \in \mathbb{R}[x]_{n,2d}(\mathcal{E})$ is non-negative using an SOS certificate, meaning that we seek a Gram matrix representation of the form
\begin{equation} \label{E:GramQscalar}
p(x) = v_d(x)^\tr \,Q \, v_d(x), \quad Q \succeq 0.
\end{equation}
Using multi-indices $\beta \in \mathbb{N}^n_d$ and $\gamma \in \mathbb{N}^n_d$ to index the entries of $Q$, the equality constraints in~\eqref{E:GramQscalar} can be rewritten as
\begin{equation*}
p(x)
= \sum_{\beta, \gamma \in \mathbb{N}^n_d} Q_{\beta, \gamma}x^{\beta + \gamma}
= \sum_{\alpha \in \mathbb{N}^n_{2d}} \left( \sum_{\beta + \gamma = \alpha}Q_{\beta, \gamma} \right) x^{\alpha}.
\end{equation*}
It is clear that, even though $p(x)$ is correlatively sparse, its Gram matrix $Q$ need not be sparse: the only requirement is that $\sum_{\beta + \gamma = \alpha}Q_{\beta, \gamma} = 0$ if  $p(x)$ does not contain the monomial $x^{\alpha}$. % Consequently, searching for an SOS decomposition of a correlatively sparse polynomial bears the same computational complexity as for dense polynomials.
Nonetheless, in order to exploit correlative sparsity and reduce the cost of searching for a suitable PSD Gram matrix, we insist that $Q$ should be sparse by imposing that
%$Q_{\beta,\gamma} = 0$ whenever there exists $(i, j) \notin  \hat{\mathcal{E}}$ such that $\beta_i + \gamma_i \geq 1$ or $\beta_j + \gamma_j \geq 1$.
$Q_{\beta,\gamma} = 0$ if the monomial $x^{\beta+\gamma}$ does not appears in any polynomial of $\mathbb{R}[x]_{n,2d}({\mathcal{E}})$. More precisely, we define
\begin{multline} \label{E:SparseScalarSOS}
        \SSOS_{n,2d}(\mathcal{E}) = \{p \in \mathbb{R}[x]_{n,2d} \mid \text{\eqref{E:GramQscalar} holds and } Q_{\beta,\gamma}=0\\
        \text{if } \exists (i, j) \notin  {\mathcal{E}} \text{ s.t. }
        \beta_i + \gamma_i \neq 0 \text{ and } \beta_j + \gamma_j \neq 0\}.
\end{multline}

Another method to exploit sparsity, proposed by Waki \textit{et al.}~\cite{waki2006sums}, is to search for a {sparse} SOS (SSOS) decomposition:
let $\mathcal{C}_1,\,\ldots,\,\mathcal{C}_t$ be the maximal cliques of ${\mathcal{G}}(\mathcal{V},{\mathcal{E}})$, let $E_{\mathcal{C}_k}$ be as in~\eqref{E:IndexMatrix} for all $k=1,\,\ldots,\,t$, and try to find PSD matrices $Q_1,\,\ldots,\,Q_t$ such that
\begin{equation}
p(x) = \sum_{k=1}^t v_d(E_{\mathcal{C}_k}x)^\tr \,Q_k \, v_d(E_{\mathcal{C}_k}x), \label{E:SSOS-waki}.
\end{equation}
In other words, one can try to write $p$ as a sum of SOS polynomials $p_k(E_{\mathcal{C}_k}x) := v_d(E_{\mathcal{C}_k}x)^\tr \,Q_k \, v_d(E_{\mathcal{C}_k}x)$, each of which depends only on the corresponding subset of variables $E_{\mathcal{C}_k}x$. Our first main result is to show that these two strategies---imposing that $Q$ is sparse according to~\eqref{E:SparseScalarSOS} and looking for the SSOS decomposition~\eqref{E:SSOS-waki}---are equivalent.

% Similar to the theorem in the matrix case, maybe we can call this result as a theorem as well
\begin{theorem} \label{T:scalarSOSDecompositionTheorem}
%We have the following equivalence
Let $\mathcal{G}(\mathcal{V},\mathcal{E})$ be a chordal graph  with maximal cliques $\{\mathcal{C}_1,\mathcal{C}_2, \ldots, \mathcal{C}_t\}$. Then, %the following equivalence
\vspace{-1mm}
    \begin{equation}
        p(x) \in \SSOS_{n,2d}(\mathcal{E}) \Leftrightarrow p(x)= \sum_{k=1}^{t} p_k(E_{\mathcal{C}_k}x),
    \end{equation}
    where $p_k(E_{\mathcal{C}_k}x)$ is an SOS polynomial in the subset of variables $E_{\mathcal{C}_k}x$.
	%$p(x) \in \SSOS_{n,2d}(\mathcal{E}) \iff$~\eqref{E:SSOS-waki} holds.
\end{theorem}

\begin{proof}
		To prove the $\Rightarrow$ part, we show that the Gram matrix $Q$ has a chordal pattern when~\eqref{E:SparseScalarSOS} holds, %when~\eqref{E:GramQscalar} and~\eqref{e:sparseQcondition} hold,
so the chordal decomposition (Theorem~\ref{T:ChordalDecompositionTheorem}) can be applied to recover~\eqref{E:SSOS-waki}. The $\Leftarrow$ part will also follows from this. %{\color{red}[Splitting the proof in steps makes it unnecessarily fragmented]}

\paragraph{$\Rightarrow$} If $\beta$ and $\gamma$ are such that an entry $Q_{\beta,\gamma}$ is not required to vanish due to~\eqref{E:SparseScalarSOS}, then $(i,j) \in {\mathcal{E}}$ for all $i,j$ such that $\beta_i+\gamma_i \neq 0$ or $\beta_j+\gamma_j \neq 0$. Consequently, the set $\mathcal{C}_{\beta,\gamma} = \{i \in \mathcal{V} \mid \beta_i+\gamma_i \neq 0\}$ is a clique of the graph $\mathcal{G}(\mathcal{V},{\mathcal{E}})$ and is therefore contained in one of its maximal cliques $\mathcal{C}_1,\,\ldots,\,\mathcal{C}_t$. In other words, $\mathcal{C}_{\beta,\gamma} \subseteq \mathcal{C}_k$ for some $k \in \{1, \ldots, t\}$.  Clearly, this also implies that
		$$
		\{i \in\mathcal{V} \mid \beta_i \neq 0\} \subseteq \mathcal{C}_k, \quad
		\{ i \in\mathcal{V} \mid \gamma_i \neq 0\} \subseteq \mathcal{C}_k.
		$$
		Thus, if $\beta$ and $\gamma$ do not satisfy the condition in~\eqref{E:SparseScalarSOS}, then there exists a value $k \in \{1, \ldots, t\}$ such that
		\begin{equation}
		\label{E:Qnonzero}
		\beta,\gamma \in \mathcal{C}_k^d := \{\alpha \in \mathbb{N}^n_d \mid \alpha_i \neq 0 \implies i \in \mathcal{C}_k\}.
		\end{equation}
		At this stage, define a hyper-graph $\mathcal{G}^d(\mathcal{V}^d,\mathcal{E}^d)$ with the multi-indices $\mathcal{V}^d = \{ \alpha \in \mathbb{N}^{n}_{d} \mid x^{\alpha} \in v_d(x)\}$ as nodes and
		\begin{equation} \label{E:SuperEdge}
		\mathcal{E}^d = \bigcup_{k=1}^t \mathcal{C}^d_k \times \mathcal{C}^d_k
		\end{equation}
		as edges. Since $\mathcal{C}_1, \ldots, \mathcal{C}_t$ are the maximal cliques of the chordal graph $\mathcal{G}(\mathcal{V}, {\mathcal{E}})$, it can be shown~\cite{zheng2018scalable} that $\mathcal{G}^d(\mathcal{V}^d, \mathcal{E}^d)$ is also chordal, and that $\mathcal{C}^d_1, \ldots , \mathcal{C}^d_t$ are its maximal cliques.
		Moreover, given that condition~\eqref{E:Qnonzero} holds for some  $k \in \{1, \ldots, t\}$ for each pair $(\beta,\gamma)$ such that $Q_{\beta,\gamma}$ is not required to vanish by~\eqref{E:SparseScalarSOS}, the hyper-graph $\mathcal{G}^d(\mathcal{V}^d, \mathcal{E}^d)$ characterizes the sparsity pattern of the PSD matrix $Q$ in~\eqref{E:GramQscalar}, \emph{i.e.}, $Q \in  \mathbb{S}^N_+(\mathcal{E}^d, 0)$ with $N = {n+d \choose d}$. Thus, according to Theorem~\ref{T:ChordalDecompositionTheorem}, $Q$ in~\eqref{E:GramQscalar} can be decomposed as
		\begin{equation} \label{E:Qdecomposition}
		Q = \sum_{k=1}^t {E}_{{\mathcal{C}}_k^d}^\tr Q_k{E}_{{\mathcal{C}}_k^d},
		\end{equation}
		where $Q_k \in \mathbb{S}_+^{|{\mathcal{C}}_k^d|}$ for all $k = 1,\,\ldots,\,t$.
		Upon noticing that ${E}_{{\mathcal{C}}_k^d}v_d(x) = v_d(E_{\mathcal{C}_k}x)$ by virtue of the definition of $\mathcal{C}_k^d$ in~\eqref{E:Qnonzero}, we then obtain from~\eqref{E:GramQscalar} that
		\begin{equation*} %\label{E:SOSdecomposition_s1}
		\begin{aligned}
		p(x) &= v_d(x)^\tr \left( \sum_{k=1}^t {E}_{{\mathcal{C}}_k^d}^\tr Q_k{E}_{{\mathcal{C}}_k^d} \right) v_d(x) \\
		& = \sum_{k=1}^t ({E}_{{\mathcal{C}}_k^d}v_d(x))^\tr Q_k ({E}_{{\mathcal{C}}_k^d}v_d(x)), \\
		& = v_d(E_{\mathcal{C}_k}x)^\tr \,Q_k \, v_d(E_{\mathcal{C}_k}x),
		%{E}_{\widetilde{\mathcal{C}}_k}^\tr Q_k{E}_{\widetilde{\mathcal{C}}_k} \left(I_r \otimes v_d(x)\right) \right],
		\end{aligned}
		\end{equation*}
		which is exactly~\eqref{E:SSOS-waki} as claimed.

	\paragraph{$\Leftarrow$}This follows after rearranging the last set of equalities in a suitable way.
	\end{proof}

Note that, in fact, we have shown that $p(x) \in \SSOS_{n,2d}$ if and only if it admits a sparse Gram matrix $Q \in  \mathbb{S}^N_+(\mathcal{E}^d, 0)$, so searching for an SSOS decomposition~\eqref{E:SSOS-waki} amounts to imposing a sparsity constraint on the Gram matrix.

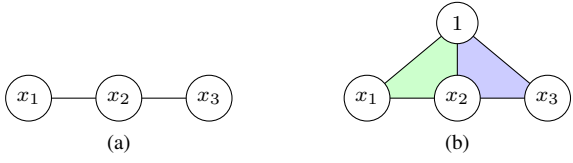
\begin{figure}[t]
    \centering
    \footnotesize
\begin{tikzpicture}

% figure (a)
% coordinates
    \coordinate (s11) at (0,-1);
    \coordinate (s12) at (1.2,-1);
    \coordinate (s13) at (2.4,-1);
    \path[draw] (s11)--(s12);
    \path[draw] (s12)--(s13);
    \node[draw, circle,color=black, fill = white] at (s11) {$x_1$};
    \node[draw, circle,color=black, fill = white] at (s12) {$x_2$};
    \node[draw, circle,color=black, fill = white] at (s13) {$x_3$};

    \node at (1.2,-1.6) {(a)};

   % figure (b)
    \coordinate (s21) at (5.7,0);
    \coordinate (s22) at (4.5,-1);
    \coordinate (s23) at (5.7,-1);
    \coordinate (s24) at (6.9,-1);
    % fill halves of triangle
   \begin{scope}[on background layer]
      \fill[green!20!white,on background layer] (s21) -- (s22) -- (s23) -- cycle;
      \fill[blue!20!white,on background layer] (s21) -- (s23) -- (s24) -- cycle;
   \end{scope}

   \path[draw] (s21)--(s22);
   \path[draw] (s21)--(s23);
   \path[draw] (s21)--(s24);
   \path[draw] (s22)--(s23);
   \path[draw] (s23)--(s24);
    \node[draw, circle, color=black,fill = white,minimum size=0.55cm] at (s21) {$1$};
    \node[draw, circle, color=black, fill=white] at (s22) {$x_1$};
    \node[draw, circle,color=black, fill=white] at (s23) {$x_2$};
     \node[draw, circle,color=black, fill=white] at (s24) {$x_3$};

   \node at (5.7,-1.6) {(b)};
\end{tikzpicture}
\caption{Graph patterns for polynomial~\eqref{E:ExampleScalar}: (a) the correlative sparsity pattern $\mathcal{G}(\mathcal{V},\mathcal{E})$ of~\eqref{E:ExampleScalar} is a line graph; (b) the corresponding super-graph $\mathcal{G}^d(\mathcal{V}^d,\mathcal{E}^d)$ is chordal with maximal cliques $\mathcal{C}_1^d = \{1,x_1,x_2\},\mathcal{C}_2^d = \{1,x_2,x_3\}$.}
    \label{F:ExampleScalar}
\end{figure}

\begin{example}
	Consider the quadratic polynomial
	\begin{align} \label{E:ExampleScalar}
	p(x) & = 2(1 + x_1 + x_3 + x_1^2 + x_1x_2 + x^2 + x_2x_3 + x_3^2)  \notag\\
	& = \begin{bmatrix} 1 \\ x_1 \\ x_2 \\ x_3\end{bmatrix}^\tr
	\underbrace{\begin{bmatrix} 2 & 1 & 0 & 1 \\
		1 & 2 & 1 & 0 \\
		0 & 1 & 2 & 1 \\
		1 & 0 & 1 & 2\end{bmatrix}}_{Q \succeq 0}\begin{bmatrix} 1 \\ x_1 \\ x_2 \\ x_3\end{bmatrix},
	\end{align}
	where ${\rm csp}(p)$ is a line graph with maximal cliques $\mathcal{C}_1=\{1,2\}$ and $\mathcal{C}_2=\{2,3\}$, as shown in Fig.~\ref{F:ExampleScalar}. The Gram matrix $Q$, which is unique in this case, is PSD and has a chordal sparsity pattern corresponding to the graph in Fig.~\ref{F:ExampleScalar}(b), and according to Theorem~\ref{T:ChordalDecompositionTheorem} it can be written as
	$$
	Q = \underbrace{\begin{bmatrix} 1 & 1 & 0 & 0 \\
		1 & 2 & 1 & 0 \\
		0 & 1 & 1 & 0 \\
		0 & 0 & 0 & 0\end{bmatrix}}_{\succeq 0} +
	\underbrace{\begin{bmatrix} 1 & 0 & 0 & 1 \\
		0  & 0 & 0 & 0 \\
		0 & 0 & 1 & 1 \\
		1 & 0 & 1 & 2\end{bmatrix}}_{\succeq 0}.
	$$
	Therefore, the 3-variate
	$p(x)$ can be written as the sum of the two bi-variate SOS polynomials: $p(x) = p_1(x_1,x_2) + p_1(x_2,x_3)$, where
	$$
	\begin{aligned}
	p_1(x_1,x_2) &= (1+x_1)^2 + (x_1+x_2)^2,\\%1 + 2x_1 + 2x_1^2 + 2x_2 + x_2^2 \\
	p_1(x_2,x_3) &= (1+x_3)^2 + (x_2+x_3)^2.%1 + 2x_3 + x_2^2 + 2x_3 + 2x_3^2 \\
	\end{aligned}
	$$
\end{example}

\vspace{2mm}

%%%%%%%%%%%%%%%%
\section{Relating SSOS to sparse DSOS/SDSOS}
\label{Section:ConnectingSSO-SDSOS}

%\subsection{Inner approximations of SOS polynomials with correlative sparsity}

We have seen that finding an SSOS decomposition of a correlatively sparse polynomial $p\in \mathbb{R}[x]_{n,2d}(\mathcal{E})$ amounts to constraining the sparsity of its Gram matrix. This fact makes it possible to draw a connection between the space $\SSOS_{n,2d}(\mathcal{E})$ of sparse SOS polynomials and those of sparse DSOS/SDSOS polynomials, defined as
\begin{align*}
\DSOS_{n,2d}(\mathcal{E}) &:= \DSOS_{n,2d} \cap \mathbb{R}[x]_{n,2d}(\mathcal{E}),\\
\SDSOS_{n,2d}(\mathcal{E}) &:= \SDSOS_{n,2d} \cap \mathbb{R}[x]_{n,2d}(\mathcal{E}).
\end{align*}
Specifically, we have the following result.

\begin{proposition} \label{Proposition:ScalarInlcusion}
	For any sparsity pattern $\mathcal{E}$
	\begin{multline}
            \DSOS_{n,2d}(\mathcal{E}) \subset \SDSOS_{n,2d}(\mathcal{E}) \\
              \subset \SSOS_{n,2d}(\mathcal{E}) \subseteq \SOS_{n,2d}(\mathcal{E}),
    \end{multline}
and the first two inclusions are strict. The third inclusion is strict unless $\mathcal{E}$ is full or $d=1$, in which cases $\SSOS_{n,2d}(\mathcal{E}) \equiv \SOS_{n,2d}(\mathcal{E})$.
\end{proposition}

\begin{proof}
	%We have $\DSOS_{n,2d}(\mathcal{E}) \subset \SDSOS_{n,2d}(\mathcal{E})$ holds since $\DSOS_{n,2d} \subset \SDSOS_{n,2d}$.
We only need to prove that $\SDSOS_{n,2d}(\mathcal{E}) \subset \SSOS_{n,2d}(\mathcal{E})$; the rest is true by definition, and the identity $\SSOS_{n,2}(\mathcal{E}) \equiv \SOS_{n,2}(\mathcal{E})$ holds because the Gram matrix representation of correlatively sparse quadratic polynomials is unique and must be sparse. Recall that any $p \in \SDSOS_{n,2d}(\mathcal{E})$ can be represented by an SDD Gram matrix $Q$, not necessarily sparse. We then construct a sparse Gram matrix $\hat{Q}$ according to
	$$
	\hat{Q}_{\beta,\gamma} = \begin{cases}
	0 & \text{if } \exists (i,j) \notin {\mathcal{E}}, \beta_i + \gamma_i \neq 0, \beta_j + \gamma_j \neq 0,\\
	Q_{\beta,\gamma} & \text{otherwise}.
	\end{cases}
	$$
	It is not difficult to see that
	$p(x) = v_d(x) \hat{Q} v_d(x)$, and that
	$\hat{Q}$ is also SDD. Indeed,
	replacing any off-diagonal entries with zeros does not affect the scaled diagonal dominance of a matrix, while
	if a diagonal entry $Q_{\beta,\beta}$ is replaced by zero, then there exists $(i,j) \notin \mathcal{E}$ such that $\beta_i \neq 0$ and $\beta_j \neq 0$, and so the entire row $Q_{\beta,\bullet}$ and column $Q_{\bullet,\beta}$ are also replaced with zeros.
	Thus, $\hat{Q}$ satisfies the condition~\eqref{E:SparseScalarSOS} and $p(x) \in \SSOS_{n,2d}(\mathcal{E})$. Finally, the inclusion $\SDSOS_{n,2d}(\mathcal{E}) \subset \SSOS_{n,2d}(\mathcal{E})$ is strict because there exist SSOS polynomials whose Gram matrix is not SDD.
\end{proof}

The implication of Proposition~\ref{Proposition:ScalarInlcusion} is simple but important: $\SSOS_{n,2d}(\mathcal{E})$ is a strictly better inner approximation of $\SOS_{n,2d}(\mathcal{E})$, compared to the DSOS/SDSOS counterparts.
%
%requiring that a correlatively sparse polynomials is DSOS/SDSOS is a stronger constraint than an SSOS condition.
%Consequently, analysis and control of sparse polynomial systems based on DSOS/SDSOS optimization typically yield more conservative results compared to using SSOS optimization.
%
%\subsection{Sparse SOS, DSOS, SDSOS optimization}
%
Consequently, given correlatively sparse polynomials {$p_0,\,\ldots,\,p_t \in \mathbb{R}[x]_{n,2d}(\mathcal{E})$} and an optimization variable $u \in \mathbb{R}^t$, the SOS optimization problem
\begin{equation}
\tag{$\mathcal{P}_{\rm sos}$}
\label{E:sparseSOS}
    \begin{aligned}
        f_{\rm sos}^* := \min_{u}\quad & w^\tr u \\[-1ex]
        \text{subject to} \quad  & p_0(x) + \sum_{i=1}^t u_ip_i(x) \in \SOS_{n,2d}(\mathcal{E}),
    \end{aligned}
\end{equation}
is better approximated if the cone $\SOS_{n,2d}(\mathcal{E})$ is replaced by $\SSOS_{n,2d}(\mathcal{E})$ instead of $\SDSOS_{n,2d}(\mathcal{E})$ or $\DSOS_{n,2d}(\mathcal{E})$. Specifically, if we denote the optimization problems arising in each of these cases by ${\mathcal{P}}_{\text{ssos}}$, ${\mathcal{P}}_{\text{sdsos}}$, and ${\mathcal{P}}_{\text{dsos}}$, and let $f^*_{\text{ssos}}$,  $f^*_{\text{sdsos}}$, and $f^*_{\text{dsos}} $ be their respective optimal values\footnote{For an infeasible problem, we denote the optimal cost value as infinity.}, Proposition~\ref{Proposition:ScalarInlcusion} implies that
\begin{equation}
f^*_{\text{dsos}} \geq  f^*_{\text{sdsos}} \geq   f^*_{\text{ssos}} \geq  f^*_{\text{sos}}.
\end{equation}

Additionally, it should be clear from Theorem~\ref{T:scalarSOSDecompositionTheorem} that the SSOS optimization problem $\mathcal{P}_{\rm ssos}$ can be recast as an SDP with multiple PSD matrix variables whose size is bounded by ${m + d \choose d}$, where $m$ is the size of the largest clique of the underlying correlative sparsity graph $\mathcal{G}(\mathcal{V},\mathcal{E})$. Therefore, even though problems ${\mathcal{P}}_{\text{dsos}}$ and ${\mathcal{P}}_{\text{sdsos}}$ can be solved as LPs and SOCPs~\cite{ahmadi2017dsos}, the added representation power offered by SSOS constraints need not add much computational cost when $m \ll n$. Table~\ref{Table:ScalarSOS} summarizes the problem types for SOS, SSOS, SDSOS, and SOS optimization. In fact, as will be demonstrated by the numerical examples in Section~\ref{Section:Results}, SSOS optimization can be much faster than DSOS/SDSOS optimization provided by the package SPOTless~\cite{tobenkin2013spotless}. Seen in this light, \textit{SSOS optimization bridges the gap between DSOS/SDSOS and SOS optimization for problems with correlatively sparse polynomials}.

\begin{table}[t]
    \centering
    \caption{Details of problem types for SOS, SSOS, SDSOS, and SOS optimization with degree $2d$ polynomials in $n$ variables. The value $m$ is the size of the largest clique of the underlying correlative sparsity graph $\mathcal{G}(\mathcal{V},\mathcal{E})$; for many problem instances, $m \ll n$.}
    \label{Table:ScalarSOS}
    \begin{tabular*}{\linewidth}{@{\extracolsep{\fill}}c r r c }
    \toprule[1pt]
     Problem & Cone & Program & Maximum PSD cone size  \\
    \midrule% \\
    ${\mathcal{P}}_{\text{sos}}$  & $\SOS_{n,2d}(\mathcal{E})$ & SDP & ${n+d \choose d}$ \\
    ${\mathcal{P}}_{\text{ssos}}$ &  $\SSOS_{n,2d}(\mathcal{E})$ & SDP  &  ${m + d \choose d}$ \\
    ${\mathcal{P}}_{\text{sdsos}}$ &  $\SDSOS_{n,2d}(\mathcal{E})$ & SOCP  & 2 \\
    ${\mathcal{P}}_{\text{dsos}}$ &  $\DSOS_{n,2d}(\mathcal{E})$ & LP &  1 \\
    \bottomrule[1pt]
    \end{tabular*}
\end{table}

%%%%%%%%%%%%%%%%
\section{Extension to sparse matrix-valued polynomials}
\label{Section:MatrixAnalysis}

The results of the previous sections can be extended to sparse matrix-valued polynomials, which arise naturally in some applications~\cite{scherer2006matrix,peet2009positive}. Let $\mathbb{R}[x]_{n,2d}^{r \times s}$ be the space of $r \times s$ matrices whose entries are polynomials in $\mathbb{R}[x]_{n,2d}$, and $\mathbb{S}[x]_{n,2d}^{r}$ be the space of $r \times r$ symmetric polynomial matrices. A symmetric polynomial matrix $P(x) \in \mathbb{S}[x]_{n,2d}^{r}$ is PSD if $P(x) \succeq 0, \forall\,x \in \mathbb{R}^n$, and it belongs to the space $\SOS_{n,2d}^{r}$ of SOS matrices if there exists $M \in \mathbb{R}[x]_{n,d}^{s \times r}$ such that $P(x) = M^\tr(x)M(x)$~\cite{gatermann2004symmetry,scherer2006matrix,kojima2003sums}.

Clearly, SOS matrices are PSD, and it is well-known (see, \emph{e.g.},~\cite{gatermann2004symmetry,scherer2006matrix,kojima2003sums}) that $P(x) \in \SOS_{n,2d}^{r}$ if and only if there exists a Gram matrix $Q \succeq 0$ such that
\begin{equation} \label{E:matrixSOS}
P(x) = \left(I_r \otimes v_d(x)\right)^\tr  Q \left(I_r \otimes v_d(x)\right),
\end{equation}
where $I_r$ is the $r\times r$ identity matrix and $\otimes$ is the usual Kronecker product.
Similarly to DSOS/SDSOS polynomials, we can define DSOS/SDSOS matrices as follows.
\begin{definition}
	A polynomial matrix $P \in \mathbb{S}[x]_{n,2d}^{r}$ is
	\begin{itemize}
		\item DSOS, denoted $P \in \DSOS_{n,2d}^r$, if it admits a Gram matrix representation~\eqref{E:matrixSOS} with a DD matrix $Q$.
		\item SDSOS, denoted $P \in \SDSOS_{n,2d}^r$, if it admits a Gram matrix representation~\eqref{E:matrixSOS} with an SDD matrix $Q$.
	\end{itemize}
\end{definition}

An alternative characterization of SOS/DSOS/SDSOS matrices can be given as follows~\cite{gatermann2004symmetry}.
\begin{proposition} \label{Proposition:Matrix&ScalarSOS}
	A polynomial matrix $P \in \mathbb{S}[x]_{n,2d}^{r}$ is  SOS (resp., DSOS or SDSOS) if and only if, given $y \in \mathbb{R}^r$, the polynomial $y^\tr P(x)y$ is SOS (resp. DSOS or SDSOS) in $[x;y] \in \mathbb{R}^{m+n}$.
\end{proposition}

\begin{proof}
	Using~\eqref{E:matrixSOS},
	$$
	\begin{aligned}
	y^\tr P(x)y & = y^\tr \left(I_r \otimes v_d(x)\right)^\tr  Q \left(I_r \otimes v_d(x)\right) y \\
	& = \left(I_r \otimes v_d(x) \cdot y \otimes 1\right)^\tr Q \left(I_r \otimes v_d(x) \cdot y \otimes 1\right) \\
	& = \left(y \otimes v_d(x) \right)^\tr Q \left(y \otimes v_d(x) \right) \\
	& = z(x,y)^\tr Q z(x,y),
	\end{aligned}
	$$
	where $z(x,y)=y \otimes v_d(x)$ is a subset of the vector of monomials in $x$ and $y$. Therefore,
	$P(x) \in \SSOS_{n,2d}^r$ (resp. $P(x) \in \DSOS_{n,2d}^r$ and $P(x) \in \SDSOS_{n,2d}^r$) if and only if $Q$ is PSD (resp., DD and SDD).
\end{proof}

\subsection{Sparse SOS, SDSOS, and DSOS matrices}

Similar to the spaces of sparse matrices mentioned in Section~\ref{Section:Preliminaries}, we can define the space of sparse symmetric polynomial matrices whose sparsity pattern is characterized by an undirected graph $\mathcal{G}(\mathcal{V},\mathcal{E})$ as
\begin{multline*}
\mathbb{S}_{n,2d}^{r}(\mathcal{E},0) := \left\{P \in \mathbb{S}[x]_{n,2d}^{r}  \mid  (i,j) \notin \mathcal{E}^* \right.\\
\left. \Rightarrow P_{ij}(x) = P_{ji}(x)=0 \right\}.
\end{multline*}
We can also introduce the subspaces of sparse SOS/SDSOS/DSOS matrices,
\begin{equation}
\begin{aligned}
\SOS_{n,2d}^r(\mathcal{E})  := \SOS_{n,2d}^r \cap \mathbb{S}_{n,2d}^{r \times r}(\mathcal{E},0), \\
\DSOS_{n,2d}^r(\mathcal{E})  := \DSOS_{n,2d}^r \cap \mathbb{S}_{n,2d}^{r \times r}(\mathcal{E},0), \\
\SDSOS_{n,2d}^r(\mathcal{E})  := \SDSOS_{n,2d}^r \cap \mathbb{S}_{n,2d}^{r \times r}(\mathcal{E},0). \\
\end{aligned}
\end{equation}

The Gram matrix representation~\eqref{E:matrixSOS} of a sparse SOS matrix $P \in \SOS_{n,2d}^r(\mathcal{E})$ can be rewritten as
$$
P(x) = \begin{bmatrix}
v_d(x)^\tr Q_{11} v_d(x) & \ldots & v_d(x)^\tr Q_{1r} v_d(x) \\
\vdots & \ddots & \vdots \\
v_d(x)^\tr Q_{r1} v_d(x) & \ldots & v_d(x)^\tr Q_{rr} v_d(x)
\end{bmatrix}
$$
where the $(i,j)$-th block $Q_{ij} \in \mathbb{S}^{N}$ of the Gram matrix is to be chosen such that $Q \succeq 0$ and
\begin{equation} \label{E:ZeroPolynomial}
v_d(x)^\tr Q_{ij} v_d(x) = p_{ij}(x)  = 0 \text{ if } (i,j) \notin \mathcal{E}^*.
\end{equation}
Note that $Q_{ij}$ need not be a zero matrix to satisfy~\eqref{E:ZeroPolynomial}, so the Gram matrix $Q$ for a sparse SOS matrix is dense in general and checking that $P \in \SOS_{n,2d}^r(\mathcal{E})$ can be computationally expensive. For this reason, in~\cite{zheng2018decomposition} the authors proposed to impose sparsity in the Gram matrix $Q$ and test if $P$ belongs to the space of polynomial matrices that admit a sparse SOS decomposition, defined as
\begin{multline} \label{E:SparseMatrixSOS}
\SSOS_{n,2d}^r(\mathcal{E}) := \left\{P \in \SOS_{n,2d}^r(\mathcal{E}) \mid  P(x) \text{~admits a Gram} \right.\\
 \left. \text{matrix~} Q \succeq 0 \text{~with~} Q_{ij} = 0 \text{~when~} p_{ij}(x) = 0 \right\}.
\end{multline}

The following proposition shows that this space is larger than both $\DSOS_{n,2d}^r(\mathcal{E})$ and $\SDSOS_{n,2d}^r(\mathcal{E})$, and is the matrix-valued analogue of Proposition~\ref{Proposition:ScalarInlcusion}.

\begin{proposition} \label{Proposition:MatrixInlcusion}
	For any pattern $\mathcal{E}$, we have
	%$DSOS_{n,2d}^r(\mathcal{E}) \subset SDSOS_{n,2d}^r(\mathcal{E})\subset SSOS_{n,2d}^r(\mathcal{E}) \subset SOS_{n,2d}^r(\mathcal{E}).$
\begin{multline}
            \DSOS_{n,2d}^r(\mathcal{E}) \subset \SDSOS_{n,2d}^r(\mathcal{E}) \\
              \subset \SSOS_{n,2d}^r(\mathcal{E}) \subseteq \SOS_{n,2d}^r(\mathcal{E}).
    \end{multline}
    and the first two inclusions are strict. The third inclusion is strict unless $\mathcal{E}$ is full in which $\SSOS_{n,2d}^r(\mathcal{E}) \equiv \SOS_{n,2d}^r(\mathcal{E})$.%or $d=1$, in which cases $\SSOS_{n,2d}^r(\mathcal{E}) \equiv \SOS_{n,2d}^r(\mathcal{E})$.
\end{proposition}
\begin{proof}
	We only need to prove that
	\begin{equation} \label{E:InclusionSDSOS}
	\SDSOS_{n,2d}^r(\mathcal{E},0) \subset \SSOS_{n,2d}^r(\mathcal{E},0),
	\end{equation}
	since the other inclusions follow directly from the definition of each space. To this end, note that any $P \in \SDSOS_{n,2d}^r(\mathcal{E})$ admits a Gram matrix representation~\eqref{E:matrixSOS} with an SDD matrix $Q$. Then, consider the matrix
	$$
	\begin{aligned}
	\hat{Q}_{ij} = \begin{cases}
	Q_{ij}, & \text{if } (i,j) \in \mathcal{E}^*,\\
	0 & \text{if } (i,j) \notin \mathcal{E}^*,
	\end{cases}
	\end{aligned}
	$$
	obtained by setting to zero blocks of $Q$ corresponding to zero entries of $P$.
	The matrix $\hat{Q}$ is still SDD, and hence PSD, and satisfies
	$$P(x) = \left(I_r \otimes v_d(x)\right)^\tr  \hat{Q} \left(I_r \otimes v_d(x)\right).$$
	Hence, $P(x) \in \SSOS_{n,2d}^r(\mathcal{E})$ and~\eqref{E:InclusionSDSOS} is true; the inclusion is strict because there clearly exist polynomial matrices in $\SSOS_{n,2d}(\mathcal{E})$ whose Gram matrix is not SDD. Finally, the identity $\SSOS_{n,2d}^r(\mathcal{E}) \equiv \SOS_{n,2d}^r(\mathcal{E})$ holds obviously when $\mathcal{E}$ contains all edges. %{\color{red}while when $d=1$ it follows from the uniqueness of the Gram matrix representation of quadratic polynomials}.
\end{proof}

As in the scalar case, sparse matrix SOS certificates are expected to be less conservative than their DSOS/SDSOS counterparts.
Additionally, if the sparsity pattern of a sparse polynomial matrix is chordal, then working with $\SSOS_{n,2d}(\mathcal{E})$ can be computationally efficient because---similar to the SSOS decomposition for scalar polynomials---it requires solving SDPs with small PSD cones. This follows from the next theorem, originally proven in~\cite{zheng2018decomposition}, which extends Theorem~\ref{T:ChordalDecompositionTheorem} to sparse polynomial matrices.
\begin{theorem}[\!\!\cite{zheng2018decomposition}]\label{T:SOSDecompositionTheorem}
	Let $\mathcal{G}(\mathcal{V},\mathcal{E})$ be a chordal graph  with maximal cliques $\{\mathcal{C}_1,\mathcal{C}_2, \ldots, \mathcal{C}_t\}$. Then, %$P(x) \in \SSOS_{n,2d}^r(\mathcal{E})$ iff there exist $P_k(x) \in \SOS_{n,2d}^{|\mathcal{C}_k|}$, $k=1,\,\ldots,\,t$, such that
\begin{equation} \label{E:DecompositionSparseCone}
P \in \SSOS_{n,2d}^r(\mathcal{E}) \Leftrightarrow P(x)= \sum_{k=1}^{t} E_{\mathcal{C}_k}^\tr P_k(x) E_{\mathcal{C}_k}
\end{equation}
with $P_k \in \SOS_{n,2d}^{|\mathcal{C}_k|}$ for each $k = 1, \ldots, t$.
\end{theorem}

\begin{proof}
A detailed proof can be found in~\cite{zheng2018decomposition}. Briefly, the ``if'' part is obvious, while the ``only if'' part relies on the fact that, when $\mathcal{G}(\mathcal{V},\mathcal{E})$ is chordal, the sparse Gram matrix $Q$ of $P$ has a chordal sparsity pattern also, and can be decomposed using Theorem~\ref{T:ChordalDecompositionTheorem} to obtain~\eqref{E:DecompositionSparseCone}.
\end{proof}

Consequently, one can use the cones $\DSOS^r_{n,2d}(\mathcal{E})$, $\SDSOS^r_{n,2d}(\mathcal{E})$, and $\SSOS^r_{n,2d}(\mathcal{E})$ to formulate increasingly better and computationally tractable approximations of large-scale matrix-valued SOS optimization problems of the form
\begin{equation}
\label{E:MatrixSOS}
    \begin{aligned}
        \min_{u}\quad & w^\tr u \\[-1ex]
        \text{subject to} \quad  & P_0(x) + \sum_{i=1}^t u_iP_i(x) \in \SOS^r_{n,2d}(\mathcal{E}),
    \end{aligned}
\end{equation}
where $P_0,\,\ldots,\,P_t \in \mathbb{S}^{r}_{n,2d}(\mathcal{E},0)$ are given sparse symmetric polynomial matrices
and $u \in \mathbb{R}^t$ is the decision variable. All comments given at the end of Section~\ref{Section:ConnectingSSO-SDSOS} on the relation between scalar SSOS/SDSOS/DSOS optimization are also valid for matrix-valued problems.

\subsection{Reduction to the scalar analysis}

Proposition~\ref{Proposition:Matrix&ScalarSOS} states that a matrix-valued polynomial $P(x)$ is SOS (resp., DSOS or SDSOS) if the associated scalar polynomial $p(x,y) = y^\tr P(x)y$ is so, and it is natural to ask if the same is true for the SSOS condition. Here show that applying Theorem~\ref{T:SOSDecompositionTheorem} to $P \in \SSOS_{n,2d}^r(\mathcal{E})$ is indeed equivalent to applying Theorem~\ref{T:scalarSOSDecompositionTheorem} to $p(x,y)$, but only if any correlative sparsity with respect to the variable $x$ is ignored.

Indeed, $p(x,y)$ has correlative sparsity pattern $\mathcal{E}$ with respect to $y$ since the monomial $y_iy_j$ appears if and only if $P_{ij}(x) \neq 0$, meaning that $(i,j) \in {\mathcal{E}}$. It is then not difficult to verify that, if any correlative sparsity with respect to the variable $x$ is ignored, Theorem~\ref{T:scalarSOSDecompositionTheorem} guarantees that $p(x,y)$ can be decomposed as
\begin{align}
p(x,y) &= \sum_{k=1}^t p_k(x,E_{\mathcal{C}_k}y)
\end{align}
for some SOS polynomials $p_k$, $k=1,\,\ldots,\,t$. In particular, each $p_k$ is quadratic in $y$ and has degree $d$ in $x$, so it admits the Gram matrix representation
\begin{equation*}
p_k(x,E_{\mathcal{C}_k}y) = [E_{\mathcal{C}_k}y\otimes v_d(x)]^\tr Q_k [E_{\mathcal{C}_k}y\otimes v_d(x)]
\end{equation*}
with $Q_k \succeq 0$. But then, upon defining
$$V_{d,k} := I_{|\mathcal{C}_k|} \otimes v_d(x)$$
and using the properties of the Kronecker product to rewrite $E_{\mathcal{C}_k}y\otimes v_d(x) = V_{d,k} E_{\mathcal{C}_k}y$, we see that
\begin{align}
%\begin{aligned}
p(x,y)
&= \sum_{k=1}^t p_k(x,E_{\mathcal{C}_k}y) \notag \\
&= \sum_{k=1}^t [E_{\mathcal{C}_k}y\otimes v_d(x)]^\tr Q_k [E_{\mathcal{C}_k}y\otimes v_d(x)] \notag\\
&= \sum_{k=1}^t (E_{\mathcal{C}_k}y)^\tr  \underbrace{V_{d,k}^\tr Q_k V_{d,k}}_{=:P_k(x)}  E_{\mathcal{C}_k}y \notag\\
&= y^\tr \left(\sum_{k=1}^t E_{\mathcal{C}_k}^\tr  P_k(x) E_{\mathcal{C}_k}\right)y.
\end{align}
Since $p(x,y) = y^\tr P(x)y$ we conclude that
\begin{equation}
P(x) = \sum_{k=1}^t E_{\mathcal{C}_k}^\tr  P_k(x) E_{\mathcal{C}_k},
\end{equation}
which is exactly the statement of Theorem~\ref{T:SOSDecompositionTheorem}. The argument can easily be reversed to show that Theorem~\ref{T:SOSDecompositionTheorem} implies the existence of an SSOS decomposition of $p(x,y)$.

\begin{remark}
It is important to note that the equivalence outlined above holds \textit{only} when any correlative sparsity of the entries of the polynomial matrix $P$ with respect to $x$ is disregarded, because we do not take it into account in our analysis of polynomial matrices. However, it may be possible to exploit correlative sparsity in $x$ when applying Theorem~\ref{T:scalarSOSDecompositionTheorem} to  $y^\tr P(x) y$. In this case, working at the scalar level will not be equivalent to applying Theorem~\ref{T:SOSDecompositionTheorem} to $P$ directly, and will instead result in a stronger (but possibly computationally cheaper) constraint. Therefore, our results for matrix-valued polynomials remain of independent interest.
\end{remark}

%%%%%%%%%%%%%%%%
\section{Numerical examples}
\label{Section:Results}

To demonstrate that DSOS/SDSOS constraints are indeed more conservative than sparse SOS conditions in practice, %, at least when the improvements of~\cite{} are not utilized,
we report the results of numerical experiments on sparse versions of the examples considered in~\cite{ahmadi2017dsos}. We implemented sparse SOS conditions in YALMIP~\cite{lofberg2004yalmip}, adapting the undocumented option \texttt{sos.csp} to exploit correlative sparsity using the chordal extension methods described in~\cite{waki2006sums}. For the DSOS/SDSOS constraints, instead, we used SPOTless~\cite{tobenkin2013spotless}. The solver MOSEK~\cite{mosek2010mosek} was used to solve the LPs, SOCPs, and SDPs arising, respectively, from DSOS, SDSOS and SSOS constraints. All computations were carried out on a PC with a 2.8 GHz Intel Core i7 CPU and 8GB of RAM.

%The resulting conic programs (LP/SOCP)

\subsection{Lower bounds on scalar polynomials}

\begin{table}%[t]
	\centering
	\caption{Optimal $\gamma$ for the SOS/SSOS/SDSOS/DSOS relaxations of problem~\eqref{E:ExamplePop_2}, as a function of the number of variables $n$.}
	\label{Table:POP_bound_1}
	\begin{tabular*}{\linewidth}{@{\extracolsep{\fill}}c| rrrrrrrr}
		\toprule[1pt]
		Dimension $n$ & 10 & 15 & 20 & 30  & 40 & 50  \\
		\midrule% \\
		${\mathcal{P}}_{\text{sos}}$  & 0.00 & 0.00 & 0.00 & * & * & * \\
		${\mathcal{P}}_{\text{ssos}}$ &  0.00 & 0.00  &  0.00 & 0.00 & 0.00 & 0.00 \\
		${\mathcal{P}}_{\text{sdsos}}$ &  44.7 & 46.0  &  46.6 & 47.2 & 44.4 & 47.5  \\
		${\mathcal{P}}_{\text{dos}}$ &  ** & **  &  ** & ** & ** & ** \\
		\bottomrule[1pt]
	\end{tabular*}
	\scriptsize
	\newline
	\raggedright
	\quad*:  Out of memory. \qquad
	**: Infeasible program.
	%\vspace{-2mm}
\end{table}

\begin{table}%[t]
	\centering
	\caption{CPU time, in seconds, required by MOSEK to solve the SOS/SSOS/SDSOS/DSOS relaxations of problem~\eqref{E:ExamplePop_1}, as a function of the number of variables $n$.}
	\label{Table:POP_time_1}
	\begin{tabular*}{\linewidth}{@{\extracolsep{\fill}}c| rrrrrrrr}
		\toprule[1pt]
		Dimension $n$ & 10 & 15 & 20 & 30  & 40 & 50  \\
		\midrule% \\
		${\mathcal{P}}_{\text{sos}}$   & 1.26 & 22.21 & 326.8 & * & * & * \\
		${\mathcal{P}}_{\text{ssos}}$  &  0.48 & 0.47  &  0.48 & 0.63 & 0.54 & 0.53 \\
		${\mathcal{P}}_{\text{sdsos}}$ &  0.69 & 1.80  &  4.96 & 25.47 & 88.50 & 232.78  \\
		${\mathcal{P}}_{\text{dsos}}$ &  ** & **  &  ** & ** & ** & ** \\
		\bottomrule[1pt]
	\end{tabular*}
	\scriptsize
	\newline
	\raggedright
	\quad*:  Out of memory. \qquad
	**: Infeasible program.
	%\;$^\dagger$: %$SSOS_{n,2d}(\mathcal{E},0)$ or $SSOS_{n,2d}^r(\mathcal{E},0)$ \\
	%$^\ddagger$: Interfaced via SPOTLESS.
	% \vspace{-2mm}
\end{table}

Given the Broyden tridiagonal polynomial
\begin{multline*}
p(x) =  ((3-2x_1)x_1 - 2x_2 + 1)^2   \\
+\sum_{i=2}^{n-1}((3-2x_i)x_i - x_{i-1} -2x_{i+1} + 1)^2  \\
+ ((3-2x_n)x_n - x_{n-1} + 1)^2,
\end{multline*}
consider the best lower bound problem
\begin{equation} \label{E:ExamplePop_1}
\begin{aligned}
\min_{\gamma} \quad & \gamma \\
\text{subject to} \quad & p(x) + \gamma x^\tr x \geq 0 \quad \forall x \in \mathbb{R}^n.
\end{aligned}
\end{equation}
Upon replacing the non-negativity constraint with an SOS/SSOS/SDSOS/DSOS conditions, this problem can be reformulated as an SDP/SDP/SOCP/LP, respectively. The optimal $\gamma$ obtained in each case for different values of $n$ is reported in Table~\ref{Table:POP_bound_1}, and MOSEK's runtime is reported in Table~\ref{Table:POP_time_1}.
For all values of $n$ the cone of DSOS polynomials is too restrictive and the DSOS constraint is infeasible. Moreover, as expected from Proposition~\ref{Proposition:ScalarInlcusion}, the SDSOS condition is more conservative that the SSOS one\footnote{For this and all other problems solved in this paper, the methods of~\cite{ahmadi2015sum,ahmadi2017optimization} are likely to improve the optimal objective value %significantly
compared to the basic SDSOS method used here, but add computational cost.}. For this example, SSOS conditions appear not to introduce any conservativeness: they yield the same optimal value as the classical SOS relaxation, and at a fraction of the computational cost. Interestingly, solving the SSOS conditions was also faster than solving SDSOS conditions. This is likely due to the fact that the SSOS condition translates to an SDP with $n-1$ PSD matrix variables of size $6 \times 6$ for this particular problem~\eqref{E:ExamplePop_1}, while, the number of second-order cones required for an SDSOS constraint is $\mathcal{O}(n^2)$. Whether sparsity can be exploited in SPOTless to formulate a smaller SOCP for sparse SDSOS constraints remains an interesting open question for future work.

%%%%%%%%%%%%%%%%%%%
\subsection{Eigenvalue bounds on matrix polynomials}

Let $\mathcal{G}(\mathcal{V},\mathcal{E})$ be the $5$-node star graph of Fig.~\ref{F:ChordalGraph}(b), and let $P \in \mathbb{S}_{2,2}^{r \times r}(\mathcal{E},0)$ be a sparse polynomial matrix whose entries are randomly generated quadratic polynomials in $2$ variables. The best lower bound on the smallest eigenvalue of $P(x)$ valid for all $x \in \mathbb{R}^5$ is given by the solution of the optimization problem
\begin{equation} \label{E:ExamplePop_2}
\begin{aligned}
\min_{\gamma}\quad & \gamma \\
\text{subject to} \quad  & P(x) + \gamma I \succeq 0 \quad \forall x \in \mathbb{R}^5.
\end{aligned}
\end{equation}
We solved this problem for $P(x)$ of increasing size $r$ after replacing the PSD constraint with SOS, SSOS, DSOS and SDSOS conditions. The optimal $\gamma$ for each case is reported in Table~\ref{Table:POP_bound_2}, while the CPU time is shown in Table~\ref{Table:POP_time_2}. As in the previous example, SSOS conditions exhibit the best trade-off between conservativeness and computational cost.

\begin{table}%[b]
	\centering
	\caption{Optimal $\gamma$ for the SOS/SSOS/SDSOS/DSOS relaxations of problem~\eqref{E:ExamplePop_2}, as a function of the matrix size $r$.}
	\label{Table:POP_bound_2}
	\begin{tabular*}{\linewidth}{@{\extracolsep{\fill}}c| rrrrrr}
		\toprule[1pt]
		 $r$ & 30 & 40  & 50 & 60 & 70 & 80  \\
		\midrule% \\
		${\mathcal{P}}_{\text{sos}}$  & 5.917 & 4.154 & 21.61 & 10.09 & 7.364 & 10.19 \\
		${\mathcal{P}}_{\text{ssos}}$ &  5.917 & 4.498 & 21.64 & 12.71 & 7.558 & 11.39\\
		${\mathcal{P}}_{\text{sdsos}}$  &  1\,254.4 & 145.5 & 762.8 & 1\,521.1 & 1\,217.3 & 598.0 \\
		${\mathcal{P}}_{\text{dsos}}$  &  ** & ** & ** & ** & ** & ** \\
		\bottomrule[1pt]
	\end{tabular*}
	\scriptsize
	\newline
	\vspace{0.1 em}
	\raggedright
	**: Infeasible program.
	%\vspace{-2mm}
\end{table}
\begin{table}%[b]
	\centering
	\caption{CPU time, in seconds, required by MOSEK to solve the SOS/SSOS/SDSOS/DSOS relaxations of problem~\eqref{E:ExamplePop_2}, as a function of the matrix size $r$.}
	\label{Table:POP_time_2}
	\begin{tabular*}{\linewidth}{@{\extracolsep{\fill}}c| rrrrrr}
		\toprule[1pt]
		 $r$  & 30 & 40  & 50 & 60 & 70 & 80  \\
		\midrule% \\
		${\mathcal{P}}_{\text{sos}}$  & 6.64 & 27.3 & 108.1 & 308.7 & 541.3 & 1\,018.6 \\
		${\mathcal{P}}_{\text{ssos}}$  &  0.35 & 0.35 & 0.33 & 0.32 & 0.32 & 0.33\\
		${\mathcal{P}}_{\text{sdsos}}$  &  1.09 & 1.29 & 2.67 & 3.70 & 4.40 & 6.02 \\
		${\mathcal{P}}_{\text{dsos}}$   &  ** & ** & ** & ** & ** & ** \\
		\bottomrule[1pt]
	\end{tabular*}
	\scriptsize
	\newline
	\raggedright
	\vspace{0.1 em}
	**: Infeasible program.
	%\vspace{-2mm}
\end{table}

%%%%%%%%%%%%%%%
\subsection{Co-positive programming}

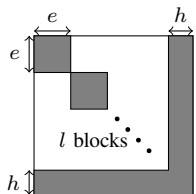
\begin{figure}%[t]%[!t]%[b]
	\centering
	\setlength{\abovecaptionskip}{1pt}
	\tikzset{decorate sep/.style 2 args=
		{decorate,decoration={shape backgrounds,shape=circle,shape size=#1,shape sep=#2}}}
	\begin{tikzpicture}[scale=0.65]
	\draw (0,0) rectangle  (3.25,3.25);
	\draw[fill=black!50] (0,3.25) rectangle  (0.75,2.5);
	\draw[fill=black!50] (0.75,2.5) rectangle  (1.5,1.75);
	\draw[fill=black!50] (0,0)--(3.25,0)--(3.25,3.25)--
	(2.75,3.25)--(2.75,0.5)--(0,0.5)--(0,0);
	\draw[decorate sep={0.5mm}{2mm},fill] (1.7,1.55)--
	node[left] {\footnotesize $l$ blocks }(2.55,0.7);
	\draw[<->] (0,3.35) -- node[above] {\footnotesize $e$} (0.75,3.35) ;
	\draw[<->] (-0.1,2.5) -- node[left] {\footnotesize $e$} (-0.1,3.25) ;
	\draw[<->] (2.75,3.35) -- node[above] {\footnotesize $h$} (3.25,3.35) ;
	\draw[<->] (-0.1,0) -- node[left] {\footnotesize $h$} (-0.1,0.5) ;
	\end{tikzpicture}
	\caption{Block-arrow sparsity pattern (dots indicate repeating diagonal blocks). The parameters are: the number of blocks, $l$; block size, $e$; the width of the arrow head, $h$.}
	\label{F:BlockArrow}
	%\vspace{-2mm}
\end{figure}

\begin{table}%[b]
	\centering
	\caption{Optimal $\gamma$ for the SOS/SSOS/SDSOS/DSOS relaxations of problem~\eqref{E:Copositive} with block size $e = 3$ and arrow head size $h = 2$, as a function of the number of blocks, $l$.}
	\label{Table:CopositiveBound}
	\begin{tabular*}{\linewidth}{@{\extracolsep{\fill}}r| rrrrrrrr}
		\toprule[1pt]
		$l$ & 2 & 4 & 6 & 8 & 10  \\
		\midrule% \\
		${\mathcal{P}}_{\text{sos}}$   & 1.137 & 4.197 & 2.836 & * & * \\
		${\mathcal{P}}_{\text{ssos}}$  &   1.137 & 4.197 & 2.836 & 4.043 & 4.718  \\
		${\mathcal{P}}_{\text{sdsos}}$   &  1.184 & 4.500  &  3.282 & 4.562 & 5.146   \\
		${\mathcal{P}}_{\text{dsos}}$  &  2.551 & 7.775  &  6.452 & 12.057 & 15.203  \\
		\bottomrule[1pt]
	\end{tabular*}
	\scriptsize
	\newline
	\raggedright
	*:  Out of memory.
	
\end{table}

\begin{table}[t]
	\centering
	\caption{CPU time, in seconds, required by MOSEK to solve the SOS/SSOS/SDSOS/DSOS relaxations of problem~\eqref{E:Copositive}. Results are given as a function of the number of blocks, $l$, for block size $e = 3$ and arrow head size $h = 2$.}
	\label{Table:CopositiveTime}
	\begin{tabular*}{\linewidth}{@{\extracolsep{\fill}}r| rrrrrrrr}
		\toprule[1pt]
		$l$ & 2 & 4 & 6 & 8 & 10  \\
		\midrule% \\
		${\mathcal{P}}_{\text{sos}}$  & 0.45 & 7.34 & 248.9 & * & * \\
		${\mathcal{P}}_{\text{ssos}}$  &  0.39 & 0.41  &  0.38 & 0.49 & 0.40  \\
		${\mathcal{P}}_{\text{sdsos}}$  &  0.54 & 1.22  &  4.99 & 11.07 & 32.18   \\
		${\mathcal{P}}_{\text{dsos}}$ &  0.59 & 0.76  &  2.19 & 5.72 & 17.11  \\
		\bottomrule[1pt]
	\end{tabular*}
	\scriptsize
	\newline
	\raggedright
	*:  Out of memory.
\end{table}

Our next experiment is an optimization problem over the cone $\mathbb{CP}^n$ of co-positive $n\times n$ matrices, which has recently attracted attention since it can model several combinatorial optimization problems exactly~\cite{dur2010copositive}. A symmetric matrix $Z \in \mathbb{S}^n$ is co-positive if $y^\tr Z y \geq 0$ for all $y \geq 0$, or~\cite{parrilo2003semidefinite}
$$
Z \in  \mathbb{CP}^n \Leftrightarrow \sum_{i,j=1}^n Z_{ij} x_i^2 x_j^2 \geq 0 \quad  \forall x\in \mathbb{R}^n.
$$
Replacing the non-negativity constraint with SOS, SSOS, SDSOS and DSOS conditions yields tractable inner approximations of $\mathbb{CP}^n$. Here, we solve such approximations for optimization problems of the form
\begin{equation} \label{E:Copositive}
\begin{aligned}
\min_{\gamma} \quad & \gamma \\
\text{subject to} \quad & Z + \gamma I \in \mathbb{CP}^n,
\end{aligned}
\end{equation}
where $Z$ is a randomly generated symmetric matrix with a block-arrow sparsity pattern with $l$ blocks of size $e\times e$, and arrow head $h$; see Fig.~\ref{F:BlockArrow} for an illustration. Such a sparsity pattern is chordal, with $l$ maximal cliques of size $e+h$. We fixed the block size $e = 3$, arrow head size $h = 2$, and varied the number of blocks $l$. Table~\ref{Table:CopositiveBound} shows that the upper bound on the optimal solution of~\eqref{E:Copositive} obtained with SSOS constraints is always strictly better than that obtained with SDSOS and DSOS optimization, and gives the same result as the classical SOS relaxation in all cases for which this could be implemented. Again, SSOS constraints are also extremely competitive in terms of CPU time, cf. Table~\ref{Table:CopositiveTime}.

\subsection{Lyapunov stability analysis} \label{Section:Lyapunov}

As our final example, we considered randomly generated $n$-dimensional, degree-3 polynomial dynamical systems of the form
\begin{equation}
\label{e:degree3SystemStructure}
\begin{cases}
\dot{x}_1  = f_1(x_1,x_2),    \\
\dot{x}_2  = f_2(x_1,x_2,x_3),  \\
\vdots \\
\dot{x}_n  = f_n(x_{n-1},x_n),  \\
\end{cases}
\end{equation}
with a linearly  stable equilibrium at the origin. We then searched for quadratic Lyapunov functions of the form
$$
V(x) = V_1(x_1,x_2) + V_2(x_1,x_2,x_3) +  \ldots + V_n(x_{n-1},x_n)
$$
that certify the nonlinear local stability of the origin for  a subset of initial conditions in the box $\mathcal{D}=[-0.1,0.1]^n$. Specifically, we looked for $V(x)$ that satisfies
%{\color{red}[Yang: what $\epsilon$ was used?]}
\begin{align*}
V(0)&=0,\\
V(x)&\geq \epsilon x^\tr x \quad \forall x \in \mathcal{D},\\
-f(x)^\tr \nabla V(x) &\geq 0 \quad \forall x\in \mathcal{D},
\end{align*}
(where we used $\epsilon = 10^{-6}$ in the simulation) after replacing the non-negativity conditions with SOS, SSOS, SDSOS, and DSOS constraints in turn. Table~\ref{Table:LyapunovTime} lists the CPU time required by MOSEK to construct suitable Lyapunov functions in each case. The classical SOS constraints could not be implemented for $n > 20 $ on our PC due to RAM limitations, while all other constraints could be implemented successfully. Although in this case all of SSOS, SDSOS and DSOS conditions are feasible, the results clearly demonstrate that SSOS are the fastest, with an approximately $200 \times$ speed improvement compared to the DSOS/SDSOS formulations set up by SPOTless when $n=50$.

\begin{table}[t]
	\centering
	\caption{CPU time, in seconds, required by MOSEK to construct a quadratic Lyapunov function for a locally stable, degree-3 polynomial system of the form~\eqref{e:degree3SystemStructure}}
	\label{Table:LyapunovTime}
	\begin{tabular*}{\linewidth}{@{\extracolsep{\fill}}r| rrrrrrrrr}
		\toprule[1pt]
		$n$ & 10 & 15 & 20 & 30 & 40  & 50 \\
		\midrule% \\
		${\mathcal{P}}_{\text{sos}}$  &1.36 & 21.26 & 262.08 & * & * & *\\
		${\mathcal{P}}_{\text{ssos}}$ &  0.57 & 0.69  &  0.76 & 1.02 & 1.22 & 1.41  \\
		${\mathcal{P}}_{\text{sdsos}}$ &  1.21 & 6.78  &  5.20 & 28.61 & 104.36 & 292.05   \\
		${\mathcal{P}}_{\text{dsos}}$ &  0.74 & 1.33 & 2.89 & 14.61 & 61.52 & 275.95  \\
		\bottomrule[1pt]
	\end{tabular*}
	\scriptsize
	\newline
	\raggedright
	*:  Out of memory.
\end{table}

%%%%%%%%%%%%%%%%
\section{Conclusion}
\label{Section:Conclusion}

In this paper we have demonstrated that, for correlatively sparse polynomials, sparse SOS positivity certificates are more general and typically less conservative than those based on DSOS and SDSOS methods. Key to this result is a new interpretation of the sparse SOS conditions proposed by Waki \textit{et al.}~\cite{waki2006sums} in terms of a sparsity constraint of the Gram matrix $Q$ used to represent sparse SOS polynomials, to which a well known chordal decomposition theorem can be applied. Numerical examples have confirmed our theoretical findings, and also demonstrated that SSOS conditions can be dramatically more efficient than the DSOS/SDSOS conditions formulated by the dedicated package SPOTless. Thus, although DSOS/SDSOS methods remain one of the few methods to implement non-negativity constraints for dense polynomials with many variables and/or high degree, one should try to utilize SSOS constraints when possible.

In the particular context of systems analysis, our findings motivate the development of robust methods that exploit sparsity in the system's governing equations and result in sparse polynomial non-negativity conditions. In Section~\ref{Section:Lyapunov}, we have done this by choosing a Lyapunov function with a structure such that the correlative sparsity of the governing equations is inherited by the eventual non-negativity constraints, but this construction was made possible by the simplicity of our example. Identifying a general procedure to formulate sparse SOS conditions will therefore be essential to enable the analysis of complex, large-scale sparse systems.

\balance
\bibliographystyle{IEEEtran}
\bibliography{Reference}

\end{document}